\documentclass[12pt]{article}
\usepackage[english]{babel}
\usepackage{amsmath,amsthm,amsfonts,amssymb,epsfig,vector}
\usepackage[left=1in,top=1in,right=1in]{geometry}

\newtheorem{thm}{Theorem}[section]
\newtheorem{lem}[thm]{Lemma}
\newtheorem{prop}[thm]{Proposition}
\newtheorem{cor}[thm]{Corollary}
\newtheorem{df}[thm]{Definition}

\newcommand{\E}{\mathbb{E}}
\newcommand{\G}{\mathcal{G}}
\newcommand{\M}{\mathcal{M}}
\newcommand{\B}{\mathcal{B}}
\newcommand{\prob}{\mathbb{P}}

\newcommand{\R}{\mathbb{R}}
\newcommand{\Z}{\mathbb{Z}}
\newcommand{\N}{\mathbb{N}}

\newcommand{\ssigma}{\vec{\sigma}}

\newcommand{\distrib}{\overset{\mbox{law}}{=}}
\newcommand{\hilbert}{\mathcal{H}}

\numberwithin{equation}{section}

\begin{document}
\date{October 12th, 2010}
\title{Short-range spin glasses \\ and Random Overlap Structures}
\author{\Large Louis-Pierre Arguin  \thanks{L.-P. Arguin is supported by the NSF grant DMS-0604869 and partially by the Hausdorff Center for Mathematics, Bonn.} \\ \small Courant Institute, NYU \\ \small New York, NY 10012, USA
\and \Large Michael Damron\thanks{M. Damron is supported by an NSF postdoctoral fellowship.} \\ \small Princeton University \\ \small Princeton, NJ 08544, USA}

\maketitle

\begin{abstract}
Properties of Random Overlap Structures (ROSt)'s constructed from the Edwards-Anderson (EA) Spin Glass model on $\Z^d$ with periodic boundary conditions are studied. 
ROSt's are $\N\times\N$ random matrices whose entries are the overlaps of spin configurations sampled from the Gibbs measure. 
Since the ROSt construction is the same for mean-field models (like the Sherrington-Kirkpatrick model) as for short-range ones (like the EA model), the setup is a good common ground to study the effect of dimensionality on the properties of the Gibbs measure.
In this spirit, it is shown, using translation invariance, that the ROSt of the EA model possesses a local stability that is stronger
than stochastic stability, a property known to hold at almost all temperatures in many spin glass models with Gaussian couplings.
This fact is used to prove stochastic stability for the EA spin glass at all temperatures and for a wide range of coupling distributions.
On the way, a theorem of Newman and Stein about the pure state decomposition of the EA model is recovered and extended.
\end{abstract}

\section{Introduction}

In this paper, we study short-range spin glasses from the point of view of Random Overlap Structures. We focus on the Edwards-Anderson model, whose definition we now state. For any finite $\Lambda\subset \Z^d$, the set of edges such that both endpoints are in $\Lambda$ will be denoted by $\Lambda^*$.
The Hamiltonian of the EA model is defined on spin configurations $\sigma \in \{-1,+1\}^\Lambda$ as
\begin{equation}
\label{eqn: H EA}
H_{\Lambda,J}=-\sum_{\{x,y\}}J_{xy}\sigma_x\sigma_y\ ,
\end{equation}
where the summation is over all nearest-neighbor edges $\{x,y\}$ in $\Lambda^*$ as well as additional edges at the boundary that endow $\Lambda$ with periodic boundary conditions. (The choice of boundary conditions will not be important until Section~\ref{subsec: TI}.)
The set $J=(J_{xy})$ for all edges $\{x,y\}$ of $\Z^d$ are the {\it couplings} of the system. 
The definition of the EA model demands that each coupling $J_{xy}$ be distributed independently under a measure on $\R$ that is symmetric.
The joint distribution of the couplings will be denoted by $\nu(dJ)$. The {\it Gibbs measure} for the volume $\Lambda$ at couplings $J$ is defined as
\begin{equation}
\label{eqn: gibbs}
G_{\Lambda,\beta,J}(\sigma)=\frac{\exp-\beta H_{\Lambda,J}(\sigma)}{Z_{\Lambda,J}(\beta)}
\end{equation}
for $Z_{\Lambda,J}(\beta)=\sum_{\sigma\in\{-1,+1\}^\Lambda}\exp\beta H_{\Lambda,J}(\sigma)$.

A fundamental question is of the nature of infinite-volume Gibbs measures, i.e., limits as $\Lambda \to {\mathbb Z}^d$ of measures~\eqref{eqn: gibbs}. It is known that for high temperatures they have a simple structure (in fact there is only one, see, e.g., \cite{Newman}), while at low temperatures it could be that there are many ``competing'' states. A metastate, a measure on infinite-volume Gibbs states, is an object used to study these states and will be a main focus of the present paper. 
We note that the treatment of the metastate construction in Section \ref{subsec:metastates} using de Finetti's theorem is non-standard and may be of independent interest.
Because it is a measure on states, the metastate contains information on the correlation functions, local functions of spins, etc. An alternative viewpoint is to consider only the {\it overlaps} of the system (here we will look at edge overlaps, but spin overlaps can be considered as well). This is the idea of the Random Overlap Structure (ROSt). Basically, ROSt's are random matrices whose entries are the overlaps of independent replicas sampled from the Gibbs measure. Thus, ROSt's record by default all information on the diversity and structure of overlaps that is relevant for the thermodynamics.

In this paper we construct a ROSt from the EA model's metastate. 
This consists of taking the thermodynamic limit of the system before measuring the overlap in an arbitrarily large window. 
(Similar measures were considered previously by Newman \& Stein and Guerra \cite{CD_MTC}.)
We show here that the metastate and the ROSt are very closely related.
As an example, we present results about stability properties of the EA ROSt. We see that from well-studied translation-invariance properties of the EA model, it follows that its metastate obeys a local stability property. This in turn implies stochastic stability of the EA ROSt. 
Stochastic stability is known to hold for a large class of models with Gaussian couplings at almost all $\beta$ \cite{AC,Parisi,Talagrand, lp_sourav, Contucci}. 
The proof presented here holds for the EA model at all dimensions, all $\beta$, and, somewhat unexpectedly, for a wide range of coupling distributions (including, e.g., ones that are continuous with support on the real line). This is a main observation of the paper. For completeness, an adaptation to the metastate setting of the standard approach to stochastic stability is given in the appendix.

In Section~\ref{subsec: pure states} we discuss the decomposition of Gibbs states in the support of the metastate into pure states and the relation between this decomposition and the representation of the ROSt on a Hilbert space. 
In particular, we use results of \cite{lp_sourav} on the support of a stochastically stable ROSt's sampling measure to deduce a recent theorem of Newman and Stein. It says that no Gibbs state in the support of the metastate can have a (non-trivial) finite decomposition into pure states. In other words, each state must be either a combination of two flip-related pure states or infinitely many. 

In Section~\ref{sec: SK}, we draw comparisons between short-range and mean-field models from the ROSt perspective. For example, although the EA model satisfies a local stability condition, the SK model does not due to the scaling of the couplings in the definition of its Hamiltonian. 
In addition, it is unclear in the EA model if two constructions of overlaps are equivalent: one where overlaps are constructed from a metastate in which
the thermodynamic limit is taken before taking the limit of the overlap, and one where overlaps are ``global'' in the sense that
both limits are taken simultaneously. Similar issues are raised in \cite{CD_NMF}.
It is shown here that in the case of the SK model, 
permutation invariance of the spins associated to correlation functions leads to equivalence of the two types of constructions.

{\bf Acknowledgments.} We thank Chuck Newman and Dan Stein for useful discussions and guidance. L.-P. Arguin gratefully acknowledges the support and hospitality of the Hausdorff Institute for Mathematics in Bonn during the Trimester Program in Stochastics, where part of this work was completed.

\section{Definitions}
\label{section: df}

\subsection{Metastates}\label{subsec:metastates}

The existence of a metastate measure on limits of Gibbs measure as treated here
corresponds to the Aizenman-Wehr construction \cite{AW}. However, it is presented 
in a different way from the standard approach, invoking in particular de Finetti's Theorem. 
The advantage is twofold: it automatically includes a result of Newman and Stein on sampled replicas using the metastate  (Proposition 7.1 of \cite{TC}), 
and it is convenient for the construction of ROSt's.

Let $W$ be a finite subset of $\Z^d$. Let $G_{\Lambda,\beta,J}$ be the Gibbs measure \eqref{eqn: gibbs} of the EA model where
it is assumed that $\Lambda\supset W$. 
We consider the pair
\begin{equation}
\label{eqn: W config}
(J, \{\sigma_W^i\}_{i\in\N})\ ,
\end{equation}
where $\{\sigma_W^i\}_{i\in\N}$ is an infinite sequence of configurations $\sigma_W\in\{-1,+1\}^W$ sampled from $G_{\Lambda,\beta,J}$ and 
$J=(J_{xy})$ are the couplings for all the edges $\{x,y\}$ of $\Z^d$.
The set of elements \eqref{eqn: W config} is therefore endowed with the measure 
$$
M_W^{\Lambda}=\nu(dJ)\prod_{i\in\N}G_{\Lambda,\beta, J}(d\sigma_W)\ .
$$ 
 It is easily seen that the family of measures $(M_W^\Lambda)_{\Lambda\in\Z^d}$ is tight for a given $W$ (the space $\{-1,+1\}^W$ is compact). 
In particular, for any sequence $\Lambda_n$ converging to $\Z^d$ there exists a subsequence $(M_W^{\Lambda_{n_k}})_k$ that converges.
In fact, by a diagonalization argument, this subsequence of $\Lambda$'s can be chosen such that $(M_W^{\Lambda_{n_k}})_k$ converges for any finite $W$.
We denote this family of limits by $(M_W)$. This family is consistent in the sense that if $W\subset W'$, $M_{W'}$ restricted to configurations in $W$ is $M_W$.
By Kolmogorov's extension theorem, we obtain a measure $M$ on 
\begin{equation}
\label{eqn: measure M}
(J, \{\sigma^i\}_{i\in\N}), \text{ where $\sigma_i\in\{-1,+1\}^{\Z^d}$ for all $i$.}
\end{equation}
We take the topology on $\{-1,+1\}^{\Z^d}$ to be the product topology.

A notion of Gibbs measure in the infinite-volume limit is recovered as follows.
For every $\Lambda$, the conditional law of $\{\sigma_W^i\}_{i\in\N}$ given $J$ under $M_W^{\Lambda}$ is exchangeable by the sampling construction,
i.e., the distribution is invariant under permutations of finitely many indices $i_1, \ldots, i_m$. 
This symmetry is preserved in the limit $\Lambda\to\Z^d$ and it follows that
the conditional law of $\{\sigma^i\}_{i\in\N}$ given $J$ under the measure $M$ is also exchangeable. The theorem of de Finetti on exchangeable sequences of random variables taking values in a Polish space (see, e.g., \cite{aldous}) guarantees the existence of a probability measure $\kappa_J$ on the space $\M(\{-1,+1\}^{\Z^d})$ of probability measures on $\{-1,+1\}^{\Z^d}$ such that
\begin{equation}
\label{eqn: measure M2}
M=\nu(dJ)\int_{\M(\{\pm 1\}^{\Z^d})} \kappa_J(d\Gamma) \prod_{i\in\N}\Gamma\ .
\end{equation}
The measure $\kappa_J$ is called a {\it metastate} of the sequence of Gibbs measures $(\G_{\Lambda,\beta,J})_{\Lambda}$. 
It must be stressed that $\kappa_J$ might not be defined for all $J$, but by conditioning, it must exist
for $\nu$-almost all $J$.
It is customary to drop the dependence on $\beta$ in the notation $\kappa_J$.
It can be proved that the metastate is supported on $\Gamma$'s that are infinite-volume Gibbs measures in the sense of the DLR equations \cite{Newman}. 
Roughly speaking, one can see a metastate as the law of the empirical measure of a sequence of spin configurations sampled from a Gibbs measure.
The above considerations can be regrouped in a theorem, which encompasses Theorem 4.2 and Proposition 7.1 of \cite{TC}.
\begin{thm}
\label{thm:meta}
Let $(\Lambda_n)$ be a sequence of finite subsets converging to $\Z^d$. Let $G_{\Lambda_n,\beta,J}$ be the Gibbs measure in $\Lambda$ with couplings $J$. Let $\nu$ denote the i.i.d. Gaussian distribution on $J=(J_{xy})$. 
There exists a subsequence $(\Lambda_{n_k})$ such that
$$
\nu(dJ)\times \prod_{i\in\N} G_{\Lambda_{n_k},\beta,J}\to \nu(dJ)\times \int_{\M(\{\pm 1\}^{\Z^d})}\kappa_J(d\Gamma)\prod_{i\in\N}\Gamma\ ,
$$
where  $\kappa_J(d\Gamma)$ is a probability measure on $\mathcal{M}(\{-1,+1\}^{\Z^d})$, the space of probability measures on $\{-1,+1\}^{\Z^d}$.
\end{thm}

The convergence in the theorem is understood in terms of finite-dimensional distributions on the couplings and on the space $\{-1,+1\}^{\Z^d}$.
A metastate for the EA model enjoys a useful covariance property in $J$ \cite{AW}. Let $J$ be a fixed set of couplings for which $\kappa_J$ is well-defined. Let $J'$ be another set of couplings such that $\kappa_{J'}$ exists and such that it equals $J$ except on a finite set of edges contained in $W\subset \Z^d$.
Write $\Delta J=(J'_{xy}-J_{x,y})_{\{x,y\}\in W^*}$.
Then the convergence of Theorem \ref{thm:meta} together with \eqref{eqn: gibbs} gives for any measurable function $F$ of the sequence $\{\sigma^i\}_{i\in\N}$
\begin{equation}
\label{eqn: covariance prop}
\int\kappa_{J'}(d\Gamma)\int F(\{\sigma^i\}) \prod_{i\in\N}d\Gamma =
\int\kappa_{J}(d\Gamma)\int F(\{\sigma^i\}) \prod_{i\in\N}d \widetilde \Gamma_i\ ,
\end{equation}
where for each $i$,
\[
\widetilde \Gamma_i(\cdot) = \frac{\Gamma(e^{\beta H_{W,\Delta J}(\sigma^i)} ~\cdot) }{\Gamma(e^{\beta H_{W,\Delta J}(\sigma^i)})}
\]
and $H_{W,\Delta J} (\sigma)= \sum_{\{x,y\} \in W^*} \Delta J_{xy} \sigma_x \sigma_y$.

\subsection{Random Overlap Structures}\label{subsec: rost}
The thermodynamic behavior of the EA model as $\beta$ is varied is intimately linked with the structure of the metastate. For this reason, it is useful to investigate this structure (e.g., on how many $\Gamma$'s is $\kappa_J$ supported, decomposition of $\Gamma$'s into extremal Gibbs measure, etc.) to determine properties of the EA model. Here we take this approach by looking more closely at the restrictions on the ROSt constructed from the metastate.

\begin{df}
\label{df: rost}
A {\it Random Overlap Structure} ({\it ROSt}) is a random $\N\times\N$-covariance matrix with $1$'s on the diagonal whose law is partially $($weakly$)$ exchangeable. In other words, if $Q=\{q_{ij}\}$ is a ROSt, then for any permutation matrix $\pi$ of a finite number of elements,
$$
\pi^{-1} Q \pi=\{q_{\pi(i)\pi(j)}\}\distrib Q\ .
$$
\end{df}
We shall write $\prob$ for the law of a ROSt and $\E$ for the expectation. Precisely, $\prob$ is defined to be a Borel probability measure on the compact Polish space of positive semi-definite symmetric ${\mathbb N} \times {\mathbb N}$ matrices with 1's on the diagonal (considered as a closed subset of $[ -1, 1]^{{\mathbb N} \times {\mathbb N}}$ equipped with the product topology). The interest in ROSt lies in the following characterization result.
\begin{thm}[Dovbysh-Sudakov \cite{DS}]
\label{thm DS}
Let $Q$ be a ROSt. There exists a random probability measure $\mu$ on the unit ball $\B$ of a Hilbert space $\hilbert$ with inner product ''$\cdot$" such that, conditionally on $\mu$,
$$
\{q_{ij}~:~ i\neq j\}\distrib \{v^i\cdot v^j~:~ i\neq j\}
$$
where $(v^i)_{i\in\N}$ are i.i.d. $\mu$-distributed and $\distrib$ denotes equality in law.
\end{thm}
\noindent
The role of $\mu$ is analogous to that of the empirical measure in the standard de Finetti's theorem for exchangeable sequences of random variables. 
We say that $\mu$ is the {\it sampling measure} of the ROSt. The sampling measure $\mu$ is not unique, however, if two sampling measures yield the same ROSt,
they differ by an isometry of $\hilbert$ that depends possibly on the realization of $\mu$ \cite{panchenko_ds}.

The law of a ROSt is determined by the finite-dimensional distributions of the entries of the matrix. Equivalently,
it is determined by the expectation over the {\it continuous functions on $s$ replicas} for every $s\in\N$.
A function $F:B^s \to\R$ mapping $\vec{v}=(v^1,...,v^s)$ to $\R$
is said to be a {\it continuous function on $s$ replicas} if it is a continous function that depends on the product
between distinct replicas, that is, of the form
\begin{equation}
\label{eqn: fct rep}
F(\vec{v})=F(v^i\cdot v^j~:~1\leq i<j\leq s)\ .
\end{equation}
In view of Theorem \ref{thm DS}, the expectation of $F$ over a ROSt $Q=\{q_{ij}\}$ then reads
$$
\E[F(q_{ij}~:~ 1\leq i<j\leq s)]=\E[\mu^{\times s}(F(v^i\cdot v^j~:~ 1\leq i<j\leq s))],
$$
where the product of $s$ copies of $\mu$ is denoted by $\mu^{\times s}$.

For spin glasses, ROSt's are useful because they encode all information on the overlaps of the system. 
By overlap, we mean a non-negative definite symmetric form on $\{-1,+1\}^{\Lambda}$. Several choices are possible. 
In this paper, we focus on {\it edge overlaps}, though a treatment for other overlaps like
spin overlaps is similar. We shall be interested in the edge overlap in the finite set $W$:
\begin{equation}
\label{eqn: local over}
R^W(\sigma,\sigma')=\frac{1}{|W^*|}\sum_{\{x,y\}\in W^*}\sigma_x\sigma_y\sigma_x'\sigma_y'\ ,
\end{equation}
where the summation is over all edges whose vertices are both contained in $W$, and for $\sigma,\sigma' \in \{-1,+1\}^{\Z^d}$, the edge overlap
\begin{equation}
\label{eqn: limit over}
R(\sigma,\sigma')=\lim_{W\to\Z^d}\frac{1}{|W^*|}\sum_{\{x,y\}\in W^*}\sigma_x\sigma_y\sigma_x'\sigma_y'\ ,
\end{equation}
whenever the limit is well-defined.
The first is a local variable and the second is a limit of local variables.

We first construct a ROSt from a metastate $\kappa_J$ of the EA model and the overlap $R^W$ as follows.
Consider the elements $(\{\sigma^i\}_{i\in\N})$ equipped with the probability measure
$$
\int_{\M(\{\pm 1\}^{\Z^d})}\kappa_J(d\Gamma)\prod_{i\in\N} \Gamma\ .
$$
We define the law of a random covariance matrix $Q^W_J=\{q^W_{ij}(J)\}$ by
\begin{equation}
\label{eqn:df Q^W}
\E\Big[F(q^W_{ij}(J)~:~1\leq i<j\leq s)\Big]=\int \kappa_J(d\Gamma) \Gamma^{\times s}\Big(F(R^W(\sigma^i,\sigma^j)~:~1\leq i<j\leq s)\Big)\ ,
\end{equation}
where $F$ is any continous function $F$ on $s$ replicas. 
In other words, the entries of $Q^W_J$ correspond to $R^W$-overlaps of replicas sampled from $\Gamma$ which in turn is sampled from the metastate $\kappa_J$.
It is also possible to define the averaged ROSt $Q^W$ from the above equations by integrating over $\nu(dJ)$. 
By construction, the law of $Q^W_J$ is weakly exchangeable. 
In particular, by Theorem \ref{thm DS}, there exists a (random) sampling measure $\mu_J^W$ on $\hilbert$
such that
\begin{equation}
\label{eqn: df mu^W}
\E \Big[F(q^W_{ij}(J)~:~1\leq i<j\leq s)\Big]=\E \Big[(\mu^W_J)^{\times s}\big(F(\vec{v})\big)\Big]\ .
\end{equation}

A ROSt for the limit $R$ of the overlap $R^W$ can  also be constructed from the metastate.
However, in this case, we need to appeal to translation invariance to ensure the existence of the limit \eqref{eqn: limit over}.
This is done in the next section.
We remark that an alternative construction of an infinite-volume ROSt would be to take $W$ growing with the size of the system $\Lambda$, for example by defining $Q^\Lambda=\{q^\Lambda_{ij}\}$ (as in \cite{lp_sourav}), 
\begin{equation}
\label{eqn:df Q^Lambda}
\E\Big[F(q^\Lambda_{ij}  ~:~1\leq i<j\leq s)\Big]=\int \nu(dJ) \ G_{\Lambda,\beta,J}^{\times s}(d\sigma)\Big(F(R^\Lambda(\sigma^i,\sigma^j)~:~1\leq i<j\leq s)\Big)\ ,
\end{equation}
and investigate the limit points of $Q^\Lambda$.
This is a procedure that could possibly lead to limits that differ from those previously constructed by taking the metastate limit $\Lambda\to\Z^d$ followed
by an overlap limit $W\to\Z^d$. We favor here the second one, since it appears to us physically more natural. As shown in Section \ref{sec: SK},
the two procedures yield the same ROSt in the mean-field case.

\section{Preliminary results}\label{subsec: TI}

From this point on, it will be assumed, unless otherwise stated, that the boundary conditions of the finite system in a box $\Lambda$ are periodic in all directions. This will imply translation invariance of infinite-volume quantities, and as a byproduct, stability of the ROSt under deterministic as well as stochastic perturbations.

Let $T=\{T_a\}$ be the translations by any vector $a$ in $\Z^d$. The operators $T_a$ act on the space of couplings and the sequence of replicas by
$$
T_a(J, \{\sigma^i\}_{i\in\N})=T_a(\{J_{xy}\}, \{\{\sigma_{x}^i\}_x\}_{i\in\N})=(\{J_{x+a,y+a}\}, \{\{\sigma_{x+a}^i\}_x\}_{i\in\N})\ .
$$
Because of the periodic boundary conditions, the measure $M$ on $(J, \{\sigma^i\}_{i\in\N})$ constructed prior to equation \eqref{eqn: measure M} is readily seen to be invariant under translation (in other words $T=\{T_a\}$ is a collection of measure-preserving maps for $M$). 
In particular, it implies that the the measure $\nu(dJ)\times\kappa_J$ on the pair $(J,\Gamma)$ is translation-invariant. (Here $T_a(\Gamma)$ is defined by $T_a(\Gamma)(A) = \Gamma(T_a(A))$ for events $A$, where $T_a(A) = \{T_a \sigma~:~\sigma\in A\}$.)
This is because $\Gamma$, being the empirical measure of $\{\sigma^i\}_{i\in\N}$, is a measurable function of $(J, \{\sigma^i\}_{i\in\N})$.

Translation invariance of the measure $M$ has two main consequences by means of Birkhoff's ergodic theorem. 
First, it provides a way to prove the existence of limits.
Typically, this is used to prove the existence of the overlap between two configurations sampled from a given measure.
Second, any $J$-measurable function that is also translation-invariant will be a constant $\nu$-almost everywhere. 
This is simply because the measure $\nu$ is ergodic, being a product measure on $J$'s.

Consequences of translation invariance have already been investigated by Newman and Stein in \cite{CD_NMF}. 
Their result is extended here to the setting of ROSt, which will be needed to prove stability of the overlap distributions in the next section.
For this purpose, we consider the edge overlap defined for $\sigma$ and $\sigma'$ in $\{-1,+1\}^{\Z^d}$ by
\begin{equation}
\label{eqn:overlap}
R(\sigma,\sigma'):=\lim_{W\to\Z^d}\frac{1}{|W^*|}\sum_{\{x,y\}\in W^*}\sigma_x\sigma_y\sigma'_x\sigma_y'\ .
\end{equation}
It is a direct consequence of the ergodic theorem that this limit exists almost surely 
under the measure $M$.
Therefore the definition of the ROSt  for the limit overlap $R$ can be done the same way as in
\eqref{eqn:df Q^W}. 
We denote the ROSt $Q_J=\{q_{ij}(J)\}$ in this case $Q_J$, and $Q$ for the averaged one. 
The definition of $Q_J$ is 
\begin{equation}
\label{eqn:df Q}
\E \Big[(F(q_{ij}(J) ~:~1\leq i<j\leq s))\Big]=\int \kappa_J(d\Gamma) \Gamma^{\times s}\Big(F(R(\sigma^i,\sigma^j)~:~1\leq i<j\leq s)\Big)\ ,
\end{equation}
and in terms of its sampling measure of Theorem \ref{thm DS} (denoted $\mu_J$),
\begin{equation}
\label{eqn:df mu}
\E \Big[\mu_J^{\times s}(F(v^i\cdot v^j~:~1\leq i<j\leq s))\Big]=\int \kappa_J(d\Gamma) \Gamma^{\times s}\Big(F(R(\sigma^i,\sigma^j)~:~1\leq i<j\leq s)\Big)\ .
\end{equation}
Since the overlap $R$ is the pointwise limit of the overlap $R^W$,
the following is straightforward from the dominated convergence theorem.
\begin{thm}
\label{thm:local}
For $\nu$-almost all $J$,
$
Q^W_J \to Q_J \text{ in law}\ .
$
\end{thm}
The main observation of \cite{CD_NMF} is that translation invariance implies that the distribution of the overlap of two replicas
does not depend on the realization of the couplings $J$, i.e., for any measurable set $A$ of $\R$ and for $\nu$-almost all $J$,
\begin{equation}
\label{eqn: NS}
\int \kappa_J(d\Gamma) \ \Gamma\times\Gamma \Big\{R(\sigma^1,\sigma^2)\in A\Big\}=\int\nu(dJ)\int \kappa_J(\Gamma) \ \Gamma\times\Gamma\Big\{R(\sigma^1,\sigma^2)\in A\Big\}\ .
\end{equation}
One way to see this is as follows. $R$ is clearly a $T$-invariant function
$$
R(T_a\sigma,T_a\sigma')=R(\sigma,\sigma').
$$
Moreover, $\Gamma\times\Gamma\Big\{R(\sigma^1,\sigma^2)\in A\Big\}$ is a measurable function of $(J,\{\sigma^i\}_{i\in\N})$ since by exchangeability
$$
\lim_{s\to\infty}\frac{1}{s}\sum_{r=1}^s 1\Big\{R(\sigma^{2r-1},\sigma^{2r})\in A\Big\}=\Gamma\times\Gamma\Big\{R(\sigma^1,\sigma^2)\in A\Big\}\ .
$$
Hence $(J, \{\sigma^i\}_{i\in\N})\mapsto\Gamma\times\Gamma\Big\{R(\sigma^1,\sigma^2)\in A\Big\}$ is a $T$-invariant function and its integral over $\kappa_J$ only depends on $J$.
Equation \eqref{eqn: NS} follows by ergodicity of the measure $\nu$.

Our first result is a straightforward generalization of the above reasoning to the distribution for an arbitrary number of replicas, in other words, to the ROSt of the EA model.
\begin{thm}
\label{thm: ti}
Let $Q_J$ be the ROSt constructed in \eqref{eqn:df Q} from the metastate $\kappa_J$ and the overlap \eqref{eqn: limit over}. Denote by $Q$ the corresponding averaged ROSt. Then
$$
Q_J= Q \text{ for $\nu$-almost all $J$.}
$$
\end{thm}
\begin{proof}
The proof is direct from the triviality of the tail field of the couplings $J$ under $\nu$, and the fact that, for any $s$ and any continuous function $F$ on $s$ replicas the function
$$
J \mapsto \int \kappa_J(d\Gamma) \Gamma^{\times s}\Big(F(R(\sigma^i,\sigma^j)~:~1\leq i<j\leq s)\Big)
$$
is a $T$-invariant function, since $R$ is.
\end{proof}

\section{Main Results}
\label{section: results}

\subsection{Stability of the EA model}
\label{subsec: stability}
We first give a simple proof of stability of the ROSt defined from the metastate of the EA model and the edge overlap $R$ under a deterministic change of a finite number of couplings. We call this property of the ROSt {\it local stability}.
Afterward we will state a main result, that local stability in fact implies stochastic stability of the ROSt, a well-studied property of the Gibbs measure of spin glasses (see e.g.
\cite{AC,Parisi,Talagrand, lp_sourav, Contucci} and \cite{CG_flip} where a variation in which couplings are flipped is studied). 
Throughout the section, it will assumed that the distribution $\nu$ on the couplings is continuous and its support is $\R$.

\begin{thm}[Local Stability]
\label{thm: ls}
Let $W\subset\Z^d$ be finite and consider a deterministic collection $J'=(J'_{xy}~:~\{x,y\}\in W^*)$. 
Define $H_W'(\sigma)=  \sum_{\{x,y\}\in W^*} J_{xy}' \ \sigma_x\sigma_y$. 
For any continuous function $F$ on the $R$-overlaps \eqref{eqn:overlap} of $s$ replicas and for $\nu$-almost all $J'$,
$$
\kappa_{J}\left[\frac{\Gamma^{\times s}(e^{\beta\sum_{i=1}^sH'_W(\sigma^i)}F(\ssigma))}{\Gamma(e^{\beta H'_W(\sigma)})^s}\right]=\kappa_{J}\left[\Gamma^{\times s}(F(\ssigma))\right]\ ,
$$
where  $F(\ssigma)=F(R(\sigma^i,\sigma^j)~:~ 1\leq i<j\leq s)$. In other words we have
$$
Q_J\distrib Q'_J
$$
where $Q_J$ is the ROSt defined by \eqref{eqn:df Q} and $Q'_J$ is defined similarly with $\Gamma$ replaced by
$$
\frac{e^{ \beta H'_W(\sigma)}\Gamma(d\sigma)}{\Gamma(e^{ \beta H'_W(\sigma)})}\ .
$$
\end{thm}
\begin{proof}
Let $F$ be a continuous function on the $R$-overlap of $s$ replicas. 
Recall that $R$ is a $T$-invariant function. 
Consider $W$, a finite subset of $\Z^d$.
We write $J(W)$ for the set of couplings that equals $J_{x,y}$ for any edge $\{x,y\}$ with at most one endpoint in $W$
and that equals $J_{xy}+ J'_{xy}$ for any edge $\{x,y\}$ with both endpoints in $W$. Since a metastate is well-defined for $\nu$-almost
all $J$, $\kappa_{J(W)}$ is well-defined for $\nu$-almost all $J'$.
By the property \eqref{eqn: covariance prop} of $\kappa_J$, we have
\begin{equation}
\label{eqn: ls1}
\kappa_{J(W)}\left[\Gamma^{\times s}(F(\ssigma))\right]=\kappa_{J}\left[\frac{\Gamma^{\times s}(e^{ \beta \sum_{i=1}^sH'_W(\sigma^i)}F(\ssigma))}{\Gamma(e^{ \beta H'_W(\sigma)})^s}\right]\ .
\end{equation}
On the other hand, it was proved in Theorem \ref{thm: ti} that the above expectation does not depend on $J$. Therefore 
\begin{equation}
\label{eqn: ls2}
\kappa_{J(W)}\left[\Gamma^{\times s}(F(\ssigma))\right]=\kappa_{J}\left[\Gamma^{\times s}(F(\ssigma))\right]\ .
\end{equation}
The claimed identity is obtained by combining \eqref{eqn: ls1} and \eqref{eqn: ls2}.
\end{proof}

Local stability of the EA model involves a local and deterministic transformation of the metastate. For stochastic stability, the transformation is random and global.
\begin{df}
\label{df: SS}
A ROSt with sampling measure $\mu$ is said to be stochastically stable if for any $\lambda>0$, $s\in\N$,
the ROSt defined from the sampling measure
\begin{equation}
\label{eqn: mu}
\frac{\mu(dv)e^{\lambda l(v)-\frac{\lambda^2}{2}\|v\|^2}}{\mu(e^{\lambda l(v)-\frac{\lambda^2}{2}\|v\|^2})}
\end{equation}
has the same law as the original ROSt,
where $(l(v),v\in\B)$ is a (isonormal) Gaussian field on $\B$ independent of $\mu$ with $E_l[l(v)l(v')]=v\cdot v'$, and $E_l$
denotes the expectation over the field $l$.
 In other words, for any continuous function $F$ on $s$ replicas,
$$
E_l \E  \left[\frac{\mu^{\times s}\Big(F(\vec{v})e^{\lambda l(v^1)-\frac{\lambda^2}{2}\|v^1\|^2} ... e^{\lambda l(v^s)-\frac{\lambda^2}{2}\|v^s\|^2}\Big)}
{\mu\Big(e^{\lambda l(v)-\frac{\lambda^2}{2}\|v\|^2}\Big)^s}\right]
=
\E\left[\mu^{\times s}\Big(F(\vec{v})\Big)\right]\ .
$$
\end{df}
A simple justification for the appearance of $e^{-\frac{\lambda^2}{2}\|v\|^2}$ in \eqref{eqn: mu} is that the expectation $E_l e^{\lambda l(v)}$ is $e^{\frac{\lambda^2}{2}(1-\|v^2\|)}$.
 It turns out that this factor is necessary for the mapping sending $\mu$ to \eqref{eqn: mu} to be well-behaved \cite{lp_sourav}.
The contribution $e^{\frac{\lambda^2}{2}}$ is irrelevant due to the normalization, and so is the term $e^{-\frac{\lambda^2}{2}\|v\|^2}$
in the case where $\mu$ is supported on a sphere.

The assumptions of the next theorem on stochastic stability of the ROSt of the EA model 
are weaker than the previously known results on stochastic stability.
For one, it holds for a large choice of coupling distributions, as opposed to Gaussian couplings only.
Second, the proof holds at any value of the Gibbs parameter $\beta$.
General results of stochastic stability are usually based on the differentiability of the free energy in the parameter $\beta$,
see, e.g., \cite{panchenko_GG2,lp_sourav}.
Thus they are usually shown to hold at almost all values of the parameter $\beta$. 
(For the SK model, it is known from the validity of the Parisi formula that the free energy is differentiable at all $\beta$ \cite{panchenko diff}.)
The approach based on the free energy applies to the setting of the ROSt defined from a metastate, however some modifications are necessary.
This is done in Appendix~\ref{appendix}.

\begin{thm}\label{thm: stochstability}
For every $\beta>0$ and for any distribution $\nu$ on the couplings that is continuous with $\R$ as support, 
the ROSt $Q_J$ defined in \eqref{eqn:df Q} from a metastate $\kappa_J$ of the EA model is stochastically stable. 
\end{thm}
\noindent
The proof of Theorem~\ref{thm: stochstability} will be given in Section~\ref{sec: proof}.
It can be extended to other coupling distributions, for instance continuous ones with bounded support. We restrict our attention to the case where the support is the real line for simplicity.

\subsection{Pure states and ROSt's}\label{subsec: pure states}

For each fixed coupling realization $J$, denote by ${\cal G}={\cal G}_J$ the set of all Gibbs measures with the EA Hamiltonian \eqref{eqn: H EA} (we will write ${\cal G}_J$ when we would like to emphasize dependence on $J$, although we suppress the appearance of $\beta$ in the notation). The set ${\cal G}$ is plainly convex. It is not difficult to see that it is compact as well (in the weak topology). Because the space of spin configurations is a metric space, Choquet's theorem applies, and each $\Gamma \in {\cal G}$ can be written as a convex combination of the extreme points of ${\cal G}$ (the {\it pure states}, written $ex({\cal G})$). In fact, more is true: each $\Gamma \in {\cal G}$ has a unique decomposition into pure states. To make this statement precise, we take the Borel $\sigma$-algebra $\widetilde {\cal F}$ on the space 
\[
{\cal M}=\mbox{ Borel measures on } \{-1,+1\}^{\Z^d}\ ,
\]
generated by the sets $S(f,{\cal B})= \{ \Gamma \in {\cal G} ~:~ \Gamma(f) \in {\cal B}\}$, for $f$ a continuous function on $\{-1,+1\}^{\Z^d}$ and ${\cal B}$ a Borel subset of $\R$. 
\begin{thm}\label{thm: decomposition}
For each $J$ and $\Gamma \in {\cal G}_J$, there exists a unique measure $m_\Gamma$ on $({\cal M},\widetilde {\cal F})$ such that $m_\Gamma(ex({\cal G}_J)) = 1$ and 
\[
\Gamma = \int_{{\cal M}} \rho ~m_{\Gamma}(d\rho)\ .
\]
\end{thm}
\begin{proof}
This theorem is a consequence of the Choquet-Meyer decomposition (see, e.g., Theorem~I.5.9 of \cite{Simon}) and the fact that the set $ex({\cal G}_J)$ is a simplex \cite[Theorem~III.2.4]{Simon}. Note that the latter theorem is stated in \cite{Simon} for translation-invariant Hamiltonians but the proof is valid for  the EA Hamiltonian with arbitrary nearest-neighbor coupling configurations $(J_{xy})$.
\end{proof}

Theorem~\ref{thm: decomposition} allows us, for a given coupling configuration $J$, to decompose Gibbs measures in the support of the metastate $\kappa_J$ into pure states. Furthermore, since the overlap $R = \lim_{W \to {\mathbb Z}^d} R^W$ exists $M$-almost surely (recall the definition of $M$ in \eqref{eqn: measure M2}), we see that for $\nu(dJ)\times \kappa_J(d\Gamma)$-almost all $\Gamma$, the associated measure $m_\Gamma$ must be supported on pure states with this same property. Using the fact that pure states have trivial tail fields (see, e.g., \cite[Theorem~III.2.5]{Simon}) and mixing properties of pure states, one may argue that for $\nu(dJ)\times\kappa_J(d\Gamma)$-almost all $\Gamma$, if we sample two pure states $\rho_1$ and $\rho_2$ independently from $m_\Gamma$, then the variable $R$ is $\rho_1 \times \rho_2$-a.s. a constant. (A sketch of the proof of this fact will be given in Appendix~\ref{sec: appb}.) We denote this value of the overlap $\rho_1 \cdot \rho_2$.

In some cases, we may draw an explicit correspondence between pure states in the support of $m_\Gamma$ (for $\Gamma$ in the support of $\kappa_J$) and vectors in the support of the sampling measure $\mu_J$ of the ROSt $Q_J$. For this purpose, we will write $Q_J$ in a slightly different way. Define the matrix-valued map $P$ on the space of $(J,\{\sigma^i\}_i)$'s by setting $P((J,\{\sigma^i\}_i)) = \widetilde Q$, where
\[
(\widetilde Q)_{i,j} = R(\sigma^i,\sigma^j)\ .
\]
This map is defined on a set of $M$-probability one. $P$ is a Borel measurable transformation to the space of covariance matrices introduced in Section~\ref{subsec: rost}. Note that for $\nu$-almost all coupling configurations $J$, the push-forward of the regular conditional probability measure $M(\cdot~|~J)$ by the map $P$ is exactly equal to the law of the ROSt $Q_J$, defined in Section~\ref{subsec: rost}. For a given value of $(J,\Gamma)$, we shall also condition on the value of $\Gamma$ and denote by ${\mathbb P}_{J,\Gamma}$ the push-forward of the conditional probability $M(\cdot~|~(J,\Gamma))$ by the map $P$. Since the conditional law of $\{\sigma^i\}_i$ is i.i.d. given $(J,\Gamma)$, it is clear that for $\nu(dJ)\times\kappa_J(d\Gamma)$-almost all $(J,\Gamma)$, the measure ${\mathbb P}_{J,\Gamma}$ is weakly exchangeable (i.e., it the law of a ROSt) and if we denote by $\mu_{J,\Gamma}$ its sampling measure, then for any continuous function $F$ of $s$ replicas, we have
\[
\E_{J,\Gamma}\left[ \mu_{J,\Gamma}^{\times s} (F(v^i \cdot v^j~:~1 \leq i<j\leq s)) \right] = \Gamma^{\times s}(F(R(\sigma^i,\sigma^j)~:~1\leq i < j\leq s))\ ,
\]
so that, if ${\mathbb P}$ is the law of $Q_J$, then
\[
\E \left[ \mu_J (F(v^i \cdot v^j~:~1 \leq i<j\leq s)) \right] = \int \kappa_J(d\Gamma) \E_{J,\Gamma}\left[ \mu_{J,\Gamma}^{\times s} (F(v^i \cdot v^j~:~1 \leq i<j\leq s)) \right]\ .
\]

The following was proved in \cite{lp_sourav}.
\begin{thm}
\label{thm:dim}
Let ${\mathbb P}$ be the law of a stochastically stable ROSt and let $\mu$ be its sampling measure. Then
\[
{\mathbb P}(\mu \mbox{ is supported on a single vector or an infinite-dimensional subset of } \B) = 1\ .
\]
\end{thm}
By Theorem~\ref{thm: stochstability}, the EA ROSt is stochastically stable at all inverse temperatures $\beta$. Therefore, we may apply Theorem~\ref{thm:dim}. It follows from the above construction that for $\nu(dJ)\times\kappa_J(d\Gamma)$-almost all $(J,\Gamma)$, the sampling measure $\mu_{J,\Gamma}$ is ${\mathbb P}_{J,\Gamma}$-a.s. supported on a single vector or an infinite-dimensional subset of $\B$.

For certain $(J,\Gamma)$'s we may construct the sampling measure explicitly. Suppose that $(J,\Gamma)$ is such that there is an integer $N$ with
\[
m_\Gamma = \sum_{i=1}^N w_i \delta_{\rho_i}\ ,
\]
for pure states $(\rho_i)$ and weights $(w_i)$ satisfying $\sum_{i=1}^N w_i=1$. Partition the set $\{\rho_i~:~1\leq i\leq N\}$ into equivalence classes of {\it congruent} pure states, i.e., use the equivalence relation
\[
\rho_i \sim \rho_j \mbox{ iff } \rho_i\cdot\rho_k = \rho_j\cdot\rho_k \mbox{ for all } 1 \leq k \leq N\ .
\]
Let $n_C$ be the number of equivalence classes, select a set of representatives $\widetilde \rho_1, \ldots, \widetilde \rho_{n_C}$ from the equivalence classes $C_1, \ldots, C_{n_C}$ and define the weights $(\widetilde w_i)$ by
\[
\widetilde w_i = \sum_{j~:~\rho_j \in C_i} w_j\ .
\]
Two pure states in different classes are called {\it incongruent}. Since the matrix $\widetilde A$, defined by
\begin{equation}\label{eq: matrixA}
(\widetilde A)_{i,j} = \widetilde \rho_i \cdot \widetilde \rho_j\ ,
\end{equation}
is positive semi-definite, it follows that we can find $n_C$ vectors $v_1, \ldots, v_{n_C}$ in $\R^{n_C}$ such that $v_i\cdot v_j = \widetilde\rho_i\cdot\widetilde\rho_j$ for all $1 \leq i\leq j\leq n_C$. It is not difficult now to see that if we create a ROSt using the sampling measure 
\[
\widetilde \mu_{J,\Gamma} = \sum_{i=1}^{n_C} \widetilde w_i \delta_{\widetilde v_i}
\]
on $\R^{n_C}$, then the law of this ROSt is the same as ${\mathbb P}_{J,\Gamma}$. Since the sampling measure of a ROSt is unique up to Hilbert space isometry (see \cite{lp_sourav}), we thus arrive at the following result of Newman and Stein \cite{NS 09}:

\begin{cor}
With $\nu(dJ) \times \kappa_J(d\Gamma)$-probability one, the state $\Gamma$ cannot be the sum of $N$ mutually incongruent pure states for $1<N<\infty$.
\end{cor}

Note that the result here is slightly stronger. One can repeat the construction above with small modifications to produce a sampling measure for each $\Gamma$ that is a countably infinite sum of pure states. In that case, the argument rules out a matrix $\widetilde A$ (from \eqref{eq: matrixA}) that has finite rank.

\section{Relation to mean-field models}\label{sec: SK}

Here we draw comparisons between the ROSt's of the EA model and the ones of the Sherrington-Kirkpatrick (SK) mean-field spin glass. The SK model is defined as follows. Fix an integer $N >0$ and let $\{J_{xy}~:~x\neq y;~  x,y \in {\mathbb N}\}$ be a family of i.i.d. mean-zero Gaussian variables with distribution $\nu$. The SK Hamiltonian is defined for spin configurations $\sigma \in \{-1,+1\}^N$ by
\begin{equation}
\label{eqn: H SK}
H_{N,J} = -\frac{1}{\sqrt N} \sum_{1 \leq x<y \leq N} J_{xy} \sigma_x \sigma_y
\end{equation}
and the corresponding measure
\begin{equation}\label{eqn: gibbs SK}
G_{N,\beta,J}(\sigma)=\frac{\exp-\beta H_{N,J}(\sigma)}{Z_{N,J}(\beta)}\ ,
\end{equation}
for $Z_{N,J}(\beta) = \sum_{\sigma\in\{-1,+1\}^N}\exp-\beta H_{N,J}(\sigma)$.
The spin overlap will be denoted by 
$$
R^N(\sigma,\sigma')=\frac{1}{N} \sum_{x=1}^N\sigma_x\sigma_x'\ .
$$
Since there is an edge between every two vertices, the edge overlap of the model is (up to a term of order $1/N$) simply half the square of the spin overlap.

The construction of a metastate of Section \ref{subsec:metastates} can be repeated {\it verbatim} for the SK model using the joint distribution of $J= \{J_{xy}~:~ x \neq y; x, y \in {\mathbb N}\}$ and spin configurations $\{\sigma_N^i\}_{i \in {\mathbb N}}$ sampled from \eqref{eqn: gibbs SK}. 
(It can also be done for other mean-field models. The reader is referred to \cite{kuelske} for examples
of the metastate framework applied to the Curie-Weiss model and the Hopfield model.)
Let $\kappa_J$ be a metastate constructed from the sequence $(G_{N,\beta,J})_N$. 
The measure 
$$
M=\nu(dJ)\times\int_{\mathcal{M}(\{\pm1\}^\N)} \kappa_J(d\Gamma)\prod_{i\in\N}\Gamma
$$
on elements $(J,\{\sigma^i\})$ clearly has stronger symmetries than the EA equivalent, for which only translation-invariance held.
Indeed, any permutation $\pi$ of finite elements of $\N$ acts on $(J,\{\sigma^i\})$ as follows
$$
\pi(J,\{\sigma^i\})= (\{J_{\pi(x)\pi(y)}\}, \{\{\sigma^i_{\pi(x)}\}_x\}_i)\ .
$$
Since the laws of the Gibbs measures $(G_{N,\beta,J})_N$  are preserved
under these transformations by the form of the Hamiltonian \eqref{eqn: H SK}, the symmetry is inherited by the measure $M$. 
This symmetry was recently exploited in \cite{panchenko_rep} to obtain a representation of the free energy of the SK model.

In the case of the SK model, the metastate is no longer supported on Gibbs measures in the sense of the DLR equations, 
as this notion is not defined for the graph with edge set $\{\{x,y\}~:~ x \neq y;~x,y \in {\mathbb N}\}$. 
For this reason, the notion of pure state is no longer well-defined. 
However, as it was discussed in Section \ref{subsec: pure states}, 
there is a connection between the pure state decomposition and
the sampling measure of the ROSt for the EA model. 
As a result, we choose to study the sampling measure of the SK model as the analogue of a pure state decomposition.

We define the ROSt $Q_J=\{q_{ij}(J)\}$ constructed from the metastate $\kappa_J$ as in \eqref{eqn:df Q} for the EA model
\begin{equation}
\label{eqn:df Q SK}
\E \Big[(F(q_{ij}(J) ~:~1\leq i<j\leq s))\Big]=\int \kappa_J(d\Gamma) \Gamma^{\times s}\Big(F(R(\sigma^i,\sigma^j)~:~1\leq i<j\leq s)\Big)\ ,
\end{equation}
where
$$
R(\sigma,\sigma')=\lim_{N\to\infty}R^N(\sigma,\sigma')\ .
$$
The above limit can be shown to exist $(\Gamma\times \Gamma\ \kappa_J(d\Gamma))$-almost surely using de Finetti's theorem and invariance under permutations of 
the distribution of the spins.
Moreover, as in the case of the EA model, the law of $Q_J$ equals the law of the $\nu$-averaged ROSt $Q$ for $\nu$-almost all $J$. 
(This is a consequence of equation \eqref{eqn: covariance prop2} below.)
It is also possible to define a limit ROSt by taking the limit of the overlap and of the Gibbs measure simultaneously
as in \eqref{eqn:df Q^Lambda}.
Let 
$Q^N=\{q^N_{ij}\}$ be the ROSt defined by
\begin{equation}
\label{eqn:df Q^N SK}
\E \Big[(F(q^N_{ij}~:~1\leq i<j\leq s))\Big]=\int \nu(dJ)\  G_{N,\beta,J}^{\times s}\Big(F(R^N(\sigma^i,\sigma^j)~:~1\leq i<j\leq s)\Big)\ .
\end{equation}
In the following proposition, we show that the mean-field nature of the model guarantees that a subsequence of $Q^N$ converges to $Q$ defined by \eqref{eqn:df Q SK}. As pointed out below equation \eqref{eqn:df Q^Lambda},
this is not necessarily the case for the EA model.
\begin{prop}
Let $(N_k)$ be a subsequence for which $(G_{N_k,\beta,J})_k$ converges to the metastate $\kappa_J$ in the sense of Theorem \ref{thm:meta}. Then the ROSt's $(Q^{N_k})$ defined by \eqref{eqn:df Q^N SK} converge in law to $Q$ defined by \eqref{eqn:df Q SK}.
\end{prop}
\begin{proof}
Let $F$ be a continuous function on $s$ replicas of the form 
$$
F(q_{ij}~:~ 1\leq i<j\leq s)= \prod_{i<j}q_{ij}^{n_{ij}}\ ,
$$
for a collection of integers $n_{ij}$. It is easily checked by, for example, the Stone-Weierstrass theorem that any continuous function on $s$ replicas can be uniformly approximated by sums of functions of this form.
We need to show that 
\begin{equation}
\label{eqn:to show}
\begin{aligned}
&\lim_{k\to\infty}\int \nu(dJ) G_{N_k,\beta,J}^{\times s}\Big(F(R^{N_k}(\sigma^i,\sigma^j)~:~1\leq i<j\leq s)\Big)=\\
&\hspace{2cm}\lim_{N\to\infty}\int \nu(dJ) \int \kappa_J(d\Gamma) \Gamma^{\times s}\Big(F(R^N(\sigma^i,\sigma^j)~:~1\leq i<j\leq s)\Big)\ .
\end{aligned}
\end{equation}
The function $F$ can be written as
\[
F(R^{N_k}(\sigma^i,\sigma^j)~:~1\leq i<j\leq s)=\prod_{i<j}\frac{1}{N_k^{n_{ij}}}\sum_{x_1,...,x_{n_{ij}}=1}^{N_k}\sigma_{x_1}^{i}\sigma_{x_1}^{j}\ldots \sigma_{x_{n_{ij}}}^{i}\sigma_{x_{n_{ij}}}^{j}
\]
\begin{equation}
\label{eqn:F exp}
= \frac{1}{N_k^T} \sum_{x_1^{1,2}, \ldots, x_{n_{1,2}}^{1,2} =1}^{N_k} \ldots \sum_{x_1^{s-1,s}, \ldots, x_{n_{s-1,s}}^{s-1,s} =1}^{N_k} \prod_{i<j} \sigma_{x_1^{i,j}}^i\sigma_{x_1^{i,j}}^j \ldots \sigma_{x_{n_{ij}}^{i,j}}^{i}\sigma_{x_{n_{ij}}^{i,j}}^j\ ,
\end{equation}
where $T=\sum_{i<j} n_{ij}$ and the iterated sums start with the pair (1,2) and include all pairs $(i,j)$ for $i<j$. We use the convention that the empty product equals 1 (in products above for which $n_{ij}=0$). We may use the invariance under permutation of spins of the measure $\int \nu(dJ) G_{N_k,\beta,J} (d\sigma)$ to simplify this expression after integration over $G_{N_k,\beta,J}^{\times s}\ \nu(dJ)$. We make a fixed choice (i.e., independent of $k$) of distinct vertices $\{x_1^{i,j}, \ldots, x_{n_{ij}}^{i,j}~:~i<j\}$ by assigning $x_1^{1,2}, \ldots, x_{n_{1,2}}^{1,2}$ to the set $\{1, \ldots, n_{1,2}\}$, assigning $x_1^{1,3}, \ldots, x_{n_{1,3}}^{1,3}$ to the set $\{n_{1,2}+1, \ldots, n_{1,2}+n_{1,3}\}$, and so on, until we assign $x_1^{s-1,s}, \ldots, x_{n_{s-1,s}}^{s-1,s}$ to the set $\{T-n_{s-1,s}, \ldots, T\}$. After straightforward combinatorics, \eqref{eqn:F exp} becomes
\begin{equation}
\label{eqn: comb1}
\begin{aligned}
&\int \nu(dJ) G_{N_k,\beta,J}^{\times s}\Big(F(R^{N_k}(\sigma^i,\sigma^j)~:~1\leq i<j\leq s)\Big)= \\
&\hspace{4cm}\int \nu(dJ) G_{N_k,\beta,J}^{\times s}\Big(\prod_{i<j}\sigma_{x_1^{i,j}}^{i}\sigma_{x_1^{i,j}}^j \ldots\sigma_{x_{n_{ij}}^{i,j}}^i\sigma_{x_{n_{ij}}^{i,j}}^j\Big)+\frac{C}{N_k},
\end{aligned}
\end{equation}
for some $C>0$, that depends on $F$ but not $k$.
This holds because the dominant term in \eqref{eqn:F exp} after integration comes when all elements of $\{x_1^{i,j}, \ldots, x_{n_{i,j}}^{i,j}~:~i<j\}$ are distinct.
The same development applies to $F(R^N(\sigma^i,\sigma^j)~:~1\leq i<j\leq s)\Big)$ under the measure $\int \kappa_J(d\Gamma) \Gamma^{\times s}\ \nu(dJ)$, since it is also permutation invariant, and one gets, for the same fixed set of distinct vertices $\{x_1^{i,j}, \ldots, x_{n_{ij}}^{i,j}~:~i<j\}$,
\begin{equation}
\label{eqn: comb2}
\begin{aligned}
&\int \nu(dJ)\int \kappa_J(d\Gamma) \Gamma^{\times s}\Big(F(R^{N}(\sigma^i,\sigma^j)~:~1\leq i<j\leq s)\Big)= \\
&\hspace{4cm} \int \nu(dJ)\int \kappa_J(d\Gamma) \Gamma^{\times s}\ \Big(\prod_{i<j}\sigma_{x_1^{i,j}}^{i}\sigma_{x_1^{i,j}}^j\ldots \sigma_{x_{n_{ij}^{i,j}}}^i\sigma_{x_{n_{ij}}^{i,j}}^j\Big)+\frac{C}{N}.
\end{aligned}
\end{equation}
Equation \eqref{eqn:to show} follows from the convergence to the metastate and by taking the limits $k\to\infty$ and $N\to\infty$ in \eqref{eqn: comb1} and \eqref{eqn: comb2} (note here that $s$, and therefore the $n_{ij}$'s, are fixed).
\end{proof}

As discussed in Section \ref{subsec: stability}, the ROSt $Q$ of the SK model is stochastically stable. However, we argue here that it cannot satisfy the property of local stability.
This is simply because the factor $1/\sqrt{N}$ appearing in \eqref{eqn: H SK} will make any local change of the couplings vanish in the limit $N\to\infty$. 
Therefore no non-trivial identity similar to \eqref{eqn: covariance prop} can be recovered.
This can be made precise as follows. Let $J'$ be a realization of the couplings that equals $0$ except on a finite set of edges between the first $M$ spins. Plainly, for any fixed $M$ and any $\epsilon>0$, $N$ can be taken large enough so that
$$
\frac{1-\epsilon}{1+\epsilon} G_{N,\beta,J}(\cdot)\leq G_{N,\beta,J+J'}(\cdot)=
\frac{G_{N,\beta,J}(e^{-\beta H_M'(\sigma)}\cdot)}{G_{N,\beta,J}(e^{-\beta H_M'(\sigma)})}\leq \frac{1+\epsilon}{1-\epsilon} {G_{N,\beta,J}(\cdot)}\ ,
$$
where $H_{M}'=-\sum_{1\leq x< y\leq M}\frac{J'_{xy}}{\sqrt{N}}\sigma_x\sigma_y$.
Since $\epsilon$ is arbitrary, an argument identical to the one leading to \eqref{eqn: covariance prop} yields
for any measurable function $F$ of $\{\sigma^i\}$ the identity
\begin{equation}
\label{eqn: covariance prop2}
\int\kappa_{J+J'}(d\Gamma)\int F(\{\sigma^i\}) \prod_{i\in\N}d\Gamma =
\int\kappa_{J}(d\Gamma)\int F(\{\sigma^i\}) \prod_{i\in\N}d\Gamma \ .
\end{equation}
It is interesting to note that, although no non-trivial transformation of the ROSt can be recovered from \eqref{eqn: covariance prop2}, 
the identity
does imply that the metastate of the SK model is invariant under local changes of couplings (since it applies to any
observable $F$ on replicas). In the EA model, the identity holds for translation-invariant observables.
It is for example not hard to see that it is not true for $F(\{\sigma^i\})=\sigma_x^1 \sigma_y^1$, with $x$ and $y$ nearest-neighbor sites, in the EA model.

\section{Proof of Theorem \ref{thm: stochstability}}
\label{sec: proof}

The identity of Theorem \ref{thm: ls} remains true if a measure on the couplings $J'$ is present, say $\nu'$:
\begin{equation}
\label{eqn: W id}
\int \nu'(dJ')\  \kappa_{J}\left[\frac{\Gamma^{\times s}\big(e^{\lambda_W \sum_{i=1}^s H_W'(\sigma^i)}F(\ssigma)\big)}{\Gamma^{\times s}(e^{\lambda_W \sum_{i=1}^s H'_W(\sigma^i)})}\right]=\kappa_{J}\left[\Gamma^{\times s}(F(\ssigma))\right]\ ,
\end{equation}
for any parameter $\lambda_W>0$ as long as $\nu'$ is absolutely continuous with respect to $\nu$. 
Take $\nu'$ to be the product measure of standard Gaussian measures.
The right-hand side is equal to $\E [\mu_J^{\times s}(F(\vec{v}))]$, by definition of the ROSt $Q_J$ and its sampling measure $\mu_J$ in  \eqref{eqn:df Q}. 
In view of Definition \ref{df: SS}, the proof reduces to establish, for an appropriate sequence $(\lambda_W~:~W\subset \Z^d)$ and for a given $\lambda>0$, the limit
\begin{equation}
\label{eqn: mu tilde}
\begin{aligned}
&\lim_{W\to\Z^d}
\int \nu'(dJ')\  \kappa_{J}\left[\frac{\Gamma^{\times s}\big(e^{\lambda_W \sum_{i=1}^s H_W'(\sigma^i)}F(\ssigma)\big)}{\Gamma^{\times s}(e^{\lambda_W \sum_{i=1}^s H'_W(\sigma^i)})}\right]\\
&\hspace{4cm}=E_l \E \left[\frac{\mu_J^{\times s}\Big(F(\vec{v})e^{\lambda l(v^1)-\frac{\lambda^2}{2}\|v^1\|^2} ... e^{\lambda l(v^s)-\frac{\lambda^2}{2}\|v^s\|^2}\Big)}{\mu_J^{\times s}\Big(e^{\lambda l(v^1)-\frac{\lambda^2}{2}\|v^1\|^2} ... e^{\lambda l(v^s)-\frac{\lambda^2}{2}\|v^s\|^2}\Big)}\right],
\end{aligned}
\end{equation}
where $l=(l(v),v\in\hilbert )$ is a Gaussian field independent of $\mu_J$, with $E_l l(v)l(v')=v\cdot v'$, and $E_l$ denotes the expectation over this field. 

A short argument motivates the limit \eqref{eqn: mu tilde}.
The identity \eqref{eqn: W id} is true for any $\lambda_W>0$.
So for $\lambda\in\R$, choose a sequence $(\lambda_W~:~W\subset \Z^d)$ such that
$\lambda_W \sqrt{|W^*|}\to \lambda$ as $W\to \Z^d$.
For this sequence,
\begin{equation}
\label{eqn: R limit}
\lim_{W\to\Z^d}\int\nu'(dJ')\ \lambda_W H'_W(\sigma)\ \lambda_W  H'_W(\sigma')=\lambda^2 R(\sigma,\sigma')\ .
\end{equation}
In particular, the variables $(\lambda_WH'_W(\sigma)~:~\sigma\in\{-1,+1\}^{W})$ under the measure $\nu'$ converge in law to Gaussian variables
$(\lambda l(\sigma)~:~\sigma\in\{-1,+1\}^{\Z^d})$ of covariance $R(\sigma,\sigma')$ wherever this overlap is well-defined. 
This is encouraging, but one has to be cautious to obtain a complete proof since the overlaps $R(\sigma,\sigma')$ 
are only defined $\Gamma\times\Gamma\ \kappa_J(d\Gamma)$ almost surely, and since the convergence to the field  $l$ is not stronger than convergence in law. 

The proof goes along the lines of the proof of the continuity of the mapping that sends $\mu$
to the modified sampling measure \eqref{eqn: mu}. 
We refer to \cite{lp_sourav} for the details on some estimates.
The idea is roughly to linearize the dependence of $\mu_J$ and $\Gamma$ in the denominator of  \eqref{eqn: mu tilde} using the weak law
of large numbers. An application of Fubini's theorem and the convergence \eqref{eqn: R limit} concludes the argument.

Consider $n$ copies $\ssigma^r$, $r=1,\ldots ,n$, of $s$ replicas each: $\ssigma=(\sigma^1,\ldots ,\sigma^s)$.
Define $G_W(\ssigma,J',\lambda_W)$ as $\prod_{i=1}^s \exp\left(\lambda_W H_W'(\sigma^i)\right)$ and
$$
F_W(\{\ssigma^r\}_{r=1}^n,J',\lambda_W \sqrt{|W^*|})=\frac{\frac{1}{n}\sum_{r=1}^n F(\ssigma^r)G_W(\ssigma^r,J',\lambda_W)}{\frac{1}{n}\sum_{r=1}^n G_W(\ssigma^r,J',\lambda_W)}\ ,
$$
where $F(\ssigma)=F(R(\sigma^i,\sigma^j)~:~ 1\leq i<j\leq s)$.
Using standard manipulations in the spirit of the weak law of large numbers \cite{lp_sourav}, it is not hard to show that there exists a $C>0$ that only depends on $F$ (in particular {\it not} on $W$) such that
\begin{equation}
\label{eqn: weak1}
\begin{aligned}
&\Big|\int \nu'(dJ') \kappa_{J}[\Gamma^{\times ns} \big(F_W(\{\ssigma^r\}_{r=1}^n,J',\lambda_W \sqrt{|W^*|})\big)]\\
&\hspace{4cm} -\int \nu'(dJ') \kappa_J\left[\frac{\Gamma^{\times s}(F(\ssigma)G_W(\ssigma,J',\lambda_W))}{\Gamma^{\times s}(G_W(\ssigma,J',\lambda_W))} \right]\Big|\leq \frac{C}{\sqrt{n}}\ .
\end{aligned}
\end{equation}
A similar estimate holds for the ROSt with sampling measure $\mu_J$. 
Let $E_z$ denote expectation over a standard Gaussian variable and $\vec{z}=(z^1,\ldots ,z^s)$. 
Write $G(\vec{v},\vec{z},l,\lambda)=\prod_{i=1}^s \exp(\lambda l(v^i)+\lambda z^i \sqrt{1-\|v^i\|^2})$, and for $n$ copies of $(\vec{v},\vec{z})$,
$$
F_{\Z^d}(\{\vec{v}^r,\vec{z}^r\}_{r=1}^n,l,\lambda)=\frac{\frac{1}{n}\sum_{r=1}^n F(\vec{v}^r)G(\vec{v}^r,\vec{z}^r,l,\lambda)}
{\frac{1}{n}\sum_{r=1}^n G(\vec{v}^r,\vec{z}^r,l,\lambda)}\ ,
$$
where $F(\vec{v})=F(v^i\cdot v^j~:~ 1\leq i<j\leq s)$.
Note that $E_z^{\times s} G(\vec{v},\vec{z},l,\lambda)=\prod_{i=1}^s \exp(\lambda l(v^i)+\frac{\lambda^2}{2}(1-\|v^i\|^2)$. In particular,
\begin{equation}
E_l \E \left[\frac{(\mu_J\times E_z)^{\times s}\big(F(\vec{v})G(\vec{v},\vec{z},l,\lambda)\big)}{(\mu_J\times E_z)^{\times s}\big(G(\vec{v},\vec{z},l,\lambda)\big)}\right]
\end{equation} 
equals the right-hand side of \eqref{eqn: mu tilde}.
The same manipulations as for \eqref{eqn: weak1} yield
\begin{equation}
\label{eqn: weak2}
\begin{aligned}
&\Big|E_l \E[(\mu_J\times E_z)^{\times ns}\big(F_{\Z^d}(\{\vec{v}^r,\vec{z}^r\}_{r=1}^n,l,\lambda)\big)]\\
&\hspace{4cm}-E_l \E \left[\frac{(\mu_J\times E_z)^{\times s}\big(F(\vec{v})G(\vec{v},\vec{z},l,\lambda)\big)}{(\mu_J\times E_z)^{\times s}\big(G(\vec{v},\vec{z},l,\lambda\big)\big)}\right]\Big|\leq \frac{C}{\sqrt{n}}\ .
\end{aligned}
\end{equation}
The important feature of the estimate \eqref{eqn: weak1} is that it holds uniformly in $W$. 
Therefore \eqref{eqn: mu tilde} can be proved by a standard $\epsilon$-argument by picking $n$ large enough once it is established that
\begin{equation}
\label{eqn: proof1}
\begin{aligned}
&\lim_{W\to\Z^d} \int \nu'(dJ') \kappa_{J}[\Gamma^{\times ns} \big(F_W(\{\ssigma^r\}_{r=1}^n,J',\lambda_W \sqrt{|W^*|})\big)]\\
&\hspace{5cm}=E_l \E[(\mu_J\times E_z)^{\times ns}\big(F_{\Z^d}(\{\vec{v}^r,\vec{z}^r\}_{r=1}^n,l,\lambda)\big)]\ .
\end{aligned}
\end{equation}

Observe that the function $F_W$ after integration over $\nu'$ is one that depends on the $R$-overlaps of $ns$ replicas (through $F(\ssigma)$) as well as the 
$R^W$-overlaps, being the covariances of the field $H_W'$. Accordingly, define the function $\overline{F}$
 \begin{equation}
\label{eqn: F}
\overline{F}(\{R^W(\sigma^{ri},\sigma^{r'j})\}; \{R(\sigma^{ri},\sigma^{r'j})\} , \lambda_W\sqrt{|W^*|}):=\int \nu'(dJ')F_W(\{\ssigma^r\}_{r=1}^s,J',\lambda_W\sqrt{|W^*|})\ ,
\end{equation}
where $r,r'$ range from $1$ to $n$ and $i,j$ from $1$ to $s$.
It is understood that the first set of coordinates of $\overline{F}$ contain the dependence on the covariances of the field $H'_W$
and the second set, the dependence on the $R$-overlaps within $F(\ssigma)$. 
Moreover, since $R(\sigma,\sigma)=R^W(\sigma,\sigma)=1$ for every $\sigma$, the dependence 
of $\overline{F}$ is only on the overlaps of distinct replicas, i.e. for $r\neq r'$ or $i\neq j$. This will be important.
It is easy to check that $\overline{F}$ is continuous in each overlap coordinate, thus bounded. 
Since $\lim_{W\to\Z^d}R^W(\sigma,\sigma')=R(\sigma,\sigma')$ for $\Gamma\times\Gamma\ \kappa_J(d\Gamma)$ almost all $\sigma$ and $\sigma'$
and $\lambda_W\sqrt{|W^*|}\to\lambda$ by definition, the dominated convergence theorem implies
\begin{equation}
\label{eqn: proof2}
\begin{aligned}
&\lim_{W\to \Z^d} \kappa_{J}[\Gamma^{\times ns} \big(\overline{F}(\{R^W(\sigma^{ri},\sigma^{r'j})\}; \{R(\sigma^{ri},\sigma^{r'j})\} ,\lambda_W\sqrt{|W^*|})]\\
&\hspace{4cm}=  \kappa_{J}[\Gamma^{\times ns} \big(\overline{F}(\{R(\sigma^{ri},\sigma^{r'j})\}; \{R(\sigma^{ri},\sigma^{r'j})\},\lambda )]\ .
\end{aligned}
\end{equation}
As it appears on the right-hand side, the function $\overline{F}$ is a continuous function on $ns$ replicas in the sense of \eqref{eqn: fct rep},
since it depends only on the $R$-overlaps of distinct replicas. By the definition of the ROSt in \eqref{eqn:df Q},
\begin{equation}
\label{eqn: proof3}
 \kappa_{J}[\Gamma^{\times ns} \big(\overline{F}(\{R(\sigma^{ri},\sigma^{r'j})\}; \{R(\sigma^{ri},\sigma^{r'j})\},\lambda )]
 =
 \E \mu_J^{\times ns}\big(\overline{F}(\{v^{ri}\cdot v^{r'j}\}; \{v^{ri}\cdot v^{r'j}\},\lambda )\big)\ .
\end{equation}
It is readily checked that the expectation $E_lE_z^{\times ns}[F_{\Z^d}(\{\vec{v}^r,\vec{z}^r\}_{r=1}^n,l,\lambda)]$
depends on $\{v^{ri}\cdot v^{r'j'}\}$ through $F(\vec{v})$ as well as through the covariances of the $ns$ Gaussian variables 
$$l(v^{ri})+z^{ri} \sqrt{1-\|v^{ri}\|^2})$$
that have variance $1$ and covariance $v^{ri}\cdot v^{r'j'}$ for $r\neq r'$ or $i\neq j$. In fact, the dependence is exactly as for $\overline{F}$. 
Successive applications of Fubini gives
\begin{equation}
\label{eqn: proof4}
 \E \mu_J^{\times ns}\big(\overline{F}(\{v^{ri},v^{r'j}\}; \{v^{ri},v^{r'j}\},\lambda )=E_l \E (\mu_J\times E_z)^{\times ns}[F_{\Z^d}(\{\vec{v}^r,\vec{z}^r\}_{r=1}^n,l,\lambda)]\ .
\end{equation}
Equations \eqref{eqn: F} $-$ \eqref{eqn: proof4} imply \eqref{eqn: proof1}, thereby concluding the proof of the theorem.

\appendix
\section{Appendix}
\label{appendix}

We present for completeness a more standard approach to prove stochastic stability (see e.g. \cite{CG,panchenko_GG2, lp_sourav}).
A slight modification of the standard proof is needed since the ROSt in \eqref{eqn:df Q} is defined as a limit of the local overlaps $R^W$, 
whereas the proof of \cite{CG,panchenko_GG2, lp_sourav} works for ROSt's defined as in \eqref{eqn:df Q^Lambda} for the overlaps of the entire system. 
Precisely, in the proof that follows, the parameter $\beta$ is varied infinitesimally within a window $W$ keeping the parameter outside fixed, whereas
in the standard proofs, it is varied throughout the system.
The argument proves stochastic stability of the ROSt at any $\beta$ for which the quenched free energy
$$
f(\beta):=\lim_{\Lambda\to \Z^d}\int \nu(dJ) \frac{1}{|\Lambda |}\log \sum_{\sigma} e^{\beta H_{\Lambda,J}(\sigma)}
$$
is differentiable. It is well-known that the infinite-volume limit
exists whenever $\Lambda\to\Z^d$ in the sense of Van Hove \cite{limit}, i.e., $\frac{|\partial \Lambda|}{|\Lambda|}\to 0$. 

\begin{thm}
\label{thm:SS}
Consider the EA model on $\Z^d$ with Gaussian couplings with distribution $\nu$.
Let $\beta>0$ be a point of differentiability of the free energy $f$ of the system.
Then the ROSt $Q_J$ constructed from a metastate $\kappa_J$ as in \eqref{eqn:df Q} at temperature $\beta$ is stochastically stable
for $\nu$-almost all $J$.
\end{thm}

\begin{proof}
Let $\Lambda\subset \Z^d$ be finite and $W\subset \Lambda$.
The Hamiltonian $H_{\Lambda,J}(\sigma)$ can be split into the Hamiltonian of the interactions strictly contained in $W$ (denoted $H_{W,J}$), 
the Hamiltonian of the ones strictly contained in $W^c$ (denoted $H_{\Lambda,W,J}$) and the interactions of 
the edges with one end point in $W$ and one in $W^c$ (denoted $V_{\partial W,J}(\sigma_W,\sigma_{W^c})$). 
If $\sigma_W$ stands for the spin configuration in $W$ and $\sigma_{W^c}$ for the one in $W^c$, the Hamiltonian becomes:
$$
H_{\Lambda,J}(\sigma)=H_{W,J}(\sigma_W)+H_{\Lambda, W,J}(\sigma_{W^c})+ V_{\partial W,J}(\sigma_W,\sigma_{W^c})\ .
$$
The parameter $\beta$ of the Gibbs measure will be fixed outside $W$ and varied inside. The notation will only keep track of the dependence on $\beta$ inside $W$.
We write $\beta_W$ for its value. The Gibbs measure will be denoted by $G_{\Lambda,\beta_W,J}$.
The variation of the parameter $\beta_W$ will be
$$\beta_W(\lambda):=\sqrt{\beta^2+\lambda^2/|W^*|}\ .$$ 
Since the couplings $J$ are assumed to have Gaussian distributions $\nu$, for any $\lambda>0$
one has the following identity in law, where $J$ and an independent copy $J'$ are distributed under $\nu$:
$$
\beta_W(\lambda) H_{W,J}(\sigma_W)\distrib \beta H_{W,J}(\sigma_W) + \frac{\lambda}{\sqrt{|W^*|}} H_{W,J'}(\sigma_W)\ .
$$
Hence
\begin{equation}
\label{eqn: id gibbs}
G_{\Lambda,\beta_W(\lambda),J}(\cdot)\distrib \frac{G_{\Lambda,\beta,J}\big(e^{\frac{\lambda}{\sqrt{|W^*|}}  H_{W,J'}(\sigma_W)}\ \cdot\big)}
{G_{\Lambda,\beta,J}\big(e^{\frac{\lambda}{\sqrt{|W^*|}}  H_{W,J'}(\sigma_W)}\big)}\ .
\end{equation}

Let $F$ denote a continuous function on $s$ replicas $\ssigma=(\sigma^1,\ldots ,\sigma^s)$. 
There exists $C>0$ such that $|F|\leq C$.
 We write for short
$$
 F_W(\ssigma)=F(R^W(\sigma^i,\sigma^j)~:~ 1\leq i<j\leq s)\ ,
$$
$$
 F(\ssigma)=F(R(\sigma^i,\sigma^j)~:~ 1\leq i<j\leq s)\ ,
$$
and
$$
F(\vec{v})=F(v^i\cdot v^j~:~ 1\leq i<j\leq s)\ ,
$$
where $\vec{v}=(v^1,\ldots , v^s)$ are vectors in a Hilbert space with inner product ``$\cdot$''.
Theorem \ref{thm:meta} and the dominated convergence theorem applied to the $R^W$-overlap give
$$
\lim_{W\to\Z^d} \lim_{\Lambda\to\Z^d} \int \nu(dJ)\ G_{\Lambda,\beta ,J}^{\times s}\big(F_W(\ssigma)\big)
= \int \nu(dJ)\  \kappa_J[\Gamma^{\times s}(F(\ssigma))]\ .
$$
In this limit, and in subsequent ones like it, the volumes $\Lambda$ are taken over a sequence used previously to construct the metastate, and 
$W$ converges to $\Z^d$ in the sense of van Hove.
By the definition of the ROSt $Q_J$ with sampling measure $\mu$ in \eqref{eqn:df Q} as well as Theorem \ref{thm: ti} 
on the independence of the ROSt on the couplings, for $\nu$-almost all $J$
$$
\int \nu(dJ)\  \kappa_J[\Gamma^{\times s}(F(\ssigma))]= \E \Big[\mu^{\times s}(F(\vec{v}))\Big]\ .
$$
According to Definition \ref{df: SS}, stochastic stability of $Q_J$ will hold if two limits are shown. First, 
\begin{equation}
\label{eqn: toshow1}
\begin{aligned}
&\lim_{W\to Z^d}\lim_{\Lambda\to \Z^d}\int \nu(dJ) G_{\Lambda,\beta_W(\lambda),J}^{\times s}(F_W(\ssigma))\\
&\hspace{2.5cm}= E_l   \E \left[\frac{\mu^{\times s}\Big(F(\vec{v})e^{\lambda l(v^1)-\frac{\lambda^2}{2}\|v^1\|^2} ... e^{\lambda l(v^s)-\frac{\lambda^2}{2}\|v^s\|^2}\Big)}{\mu^{\times s}\Big(e^{\lambda l(v^1)-\frac{\lambda^2}{2}\|v^1\|^2} ... e^{\lambda l(v^s)-\frac{\lambda^2}{2}\|v^s\|^2}\Big)}\right] \text{ for $\nu$-a.e. $J$ ,}
\end{aligned}
\end{equation}
where $l=(l(v)~:~v\in\B)$ is a Gaussian field on $\B\subset \hilbert$ independent of $\mu_J$, with $E_l l(v)l(v')=v\cdot v'$, and $E_l$ denotes the expectation over this field.
And second,
\begin{equation}
\label{eqn: toshow2}
\lim_{W\to Z^d}\lim_{\Lambda\to \Z^d}  \Big| \int \nu(dJ) G_{\Lambda,\beta_W(\lambda),J}^{\times s}(F_W(\ssigma)) 
-  \int \nu(dJ)\ G_{\Lambda,\beta ,J}^{\times s}\big(F_W(\ssigma)\big)\Big |=0\ .
\end{equation}
Equation \eqref{eqn: toshow1} follows from identity \eqref{eqn: id gibbs} and the same approximation scheme that proved \eqref{eqn: mu tilde}.
We focus on showing \eqref{eqn: toshow2}.

Straightforward differentiation yields
$$
\begin{aligned}
& \partial_{\beta_W} \int \nu(dJ) \ \G_{\Lambda,\beta_W,J}^{\times s}(F_W(\ssigma))\\
&\hspace{2cm}= s \  \int \nu(dJ)\ \G_{\Lambda, \beta_W,J}^{\times s}\Big[\Big\{H_{W,J}(\sigma_W^1)-G_{\Lambda, \beta_W,J}\big(H_W(\sigma_W^1)\big)\Big\}F_W(\ssigma)\Big]\ .
\end{aligned}
$$
 Hence the absolute value of the derivative is bounded by 
 $$s C \ \int\nu(dJ) \ \G_{\Lambda,\beta_W,J}\left|H_{W,J}(\sigma_W)-\G_{\Lambda,\beta_W,J}\big(H_{W,J}(\sigma_W)\big)\right|\ .$$
 By integration, we have the bound
 \begin{multline*}
 \left|\int\nu(dJ)\ \mathcal{G}_{\Lambda, \beta_W(\lambda), J}^{\times s }\big(F_W(\ssigma)\big)- \int\nu(dJ) \ \G_{\Lambda, \beta, J}^{\times s }\big(F_W(\ssigma)\big)\right|\\
 \leq sC \int_{\beta}^{\beta_W(\lambda)}d\beta'_W \int\nu(dJ)\ \G_{\Lambda, \beta_W',J}\left|H_{W,J}(\sigma_W)-\G_{\Lambda, \beta_W',J}(H_{W,J}(\sigma_W))\right|\ .
 \end{multline*}
The integral term goes to zero in the limit. Indeed, by the Cauchy-Schwarz inequality,
 \begin{multline*}
 \int \nu(dJ)\int_{\beta}^{\beta_W(\lambda)}d\beta'_W\ \mathcal{G}_{\Lambda, \beta'_W}\left|H_W(\sigma)-\mathcal{G}_{\Lambda, \beta'_W}(H_W(\sigma))\right|
 \leq \\
 \sqrt{|W|\big(\beta_W(\lambda)-\beta\big)}\left(\int \nu(dJ)\int_\beta^{\beta_W(\lambda)}d\beta'_W\frac{1}{|W|}\ \mathcal{G}_{\Lambda, \beta'_W,J}\left|H_W(\sigma)-\mathcal{G}_{\Lambda, \beta'_W,J}(H_W(\sigma))\right|^2 \right)^{1/2}\ .
 \end{multline*}
Since $\beta_W(\lambda)-\beta=\frac{\lambda^2}{2\beta |W^*|}+ O(\frac{1}{|W^*|^2})$, 
and $|W^*|/|W|$ is bounded in $\Z^d$, the claim would follow if the term in the parentheses goes to zero as $\Lambda\to\Z^d$
followed by $W\to\Z^d$.
Define 
$$
f_{\Lambda,W}(\beta_W):= \frac{1}{|W|}\int \nu(dJ) \log \frac{\sum_{\sigma_W,\sigma_{W^c}} e^{\beta_W H_{W,J}(\sigma_W)}e^{\beta H_{\Lambda, W,J}(\sigma_{W^c})}e^{\beta V_{\partial W,J}(\sigma_W,\sigma_{W^c})}}{\sum_{\sigma_{W^c}}e^{\beta H_{\Lambda, W,J}(\sigma_{W^c})}}\ .
$$
The second derivative $f''_{\Lambda, W}(\beta_W)$ is exactly the term appearing in the integral of the parentheses. 
In particular, the integral in the parentheses equals
$$
f'_{\Lambda, W}(\beta_W(\lambda))-f'_{\Lambda, W}(\beta)\ .
$$
We claim that 
$$
\lim_{W\to \Z^d} \limsup_{\Lambda\to \Z^d} |f'_{\Lambda, W}(\beta_W(\lambda))-f'_{\Lambda, W}(\beta)|=0
$$
Since $\beta_W(\lambda)\to \beta$ as $W\to\Z^d$ and $\beta$ is a point of differentiability of $f$, this will hold by a standard result of convexity (see, e.g., p.94 in \cite{Tal_book}) provided that for all $\beta_W>0$
\begin{equation}
\label{eqn: claim f}
\lim_{W\to \Z^d} \limsup_{\Lambda\to \Z^d} f_{\Lambda, W}(\beta_W)= f(\beta_W) \ .
\end{equation}
This is expected since the free energy $f$ is typically not sensitive to the boundary conditions imposed by $\Lambda$ on $W$. We make this precise.
Let
$$
\begin{aligned}
f_W(\beta,W)&:= 
 \frac{1}{|W|}\int \nu(dJ) \log \sum_{\sigma_W} e^{\beta_W H_{W,J}(\sigma_W)}\\
 &= \frac{1}{|W|}\int \nu(dJ) \log \frac{\sum_{\sigma_W,\sigma_{W^c}} e^{\beta_W H_{W,J}(\sigma_W)}e^{\beta H_{\Lambda, W,J}(\sigma_{W^c})}}{\sum_{\sigma_{W^c}}e^{\beta H_{\Lambda, W,J}(\sigma_{W^c})}}\ ,
\end{aligned}
$$
where the second equality holds trivially for any choice of $\Lambda\supset W$ and $\sigma_{W^c}$ is a spin configuration in $\Lambda \backslash W$.
Remark that because the couplings are Gaussians
$$\int \nu(dJ) e^{\beta_W V_{\partial W,J}(\sigma_W,\sigma_{W^c})}= e^{\frac{\beta_W^2}{2}\ |\partial W|}\ ,$$
where $|\partial W| $ denotes the number of edges with one endpoint in $W$ and one in $W^c$.
Since the couplings in $V_{\partial W, J}$ are independent of the ones in $W$ and $W^c$, Jensen's inequality applied to the logarithm in $f_{\Lambda,W}$ and through the integration over these couplings gives
\begin{equation}
\label{eqn: estimate f}
0 \leq f_{\Lambda,W}(\beta_W)-f_W(\beta_W)\leq \frac{\beta_W^2}{2}\frac{|\partial W|}{|W|}\ .
\end{equation}
Morover it is clear that
\begin{equation}
\label{eqn: limit f}
f(\beta_W)=\lim_{W\to \Z^d}f_W(\beta_W)\ .
\end{equation}
In view of \eqref{eqn: estimate f} 
and \eqref{eqn: limit f},
$$
\lim_{W\to\Z^d}\sup_{\Lambda\supset W} |f_{\Lambda,W}(\beta_W)- f_W(\beta_W)|=0\ ,
$$
and \eqref{eqn: claim f} follows.

\end{proof}

\section{Overlaps for pure states}\label{sec: appb}

We first recall the definition of $m_\Gamma$ from Section~\ref{subsec: pure states}. Given $\Gamma \in {\cal G}_J$, the set of Gibbs measures for the EA Hamiltonian and coupling configuration $J$, $m_\Gamma$ is the unique Borel measure on ${\cal M}$, the set of Borel measures on $\{-1,+1\}^{\Z^d}$ such that $m_\Gamma(ex({\cal G}_J))=1$ and
\[
\Gamma = \int_{ex({\cal G}_J)} \rho m_\Gamma(d\rho)\ .
\]
Again, we will drop the dependence on $J$ in the notation ${\cal G}_J$ when it is obvious from the context.

In this section we will prove that overlaps between pure states sampled from the measure $m_\Gamma$ for $\Gamma$ in the support of $\kappa_J$ are a.s. constant. Because the arguments are mostly from \cite{Simon}, we will only give the main ideas. For notational convenience, we will prove the statement for spin overlaps, but essentially the same argument holds for edge overlaps. We will do this by showing that the variance of $R$ with respect to the measure $\rho_1\times\rho_2$ (for pure states $\rho_1$ and $\rho_2$) is 0. We first list some properties of pure states and of the Choquet map $\Gamma \mapsto m_\Gamma$. Let $T_\infty$ be the tail field of the space of spin configurations. 
\begin{lem}\label{lem: simon}
Let $\rho\in ex({\cal G})$.
\begin{enumerate}
\item For each $A \in T_\infty$,
\[
\rho(A) \in \{0,1\}\ .
\]
\item For each fixed $x\in \Z^d$,
\[
\lim_{|y-x|\to \infty} |\rho(\sigma_x\sigma_y)-\rho(\sigma_x)\rho(\sigma_y)| = 0\ .
\]
\end{enumerate}
\end{lem}
\begin{proof}
The first part of the lemma is proved in Theorem~III.2.5 of \cite{Simon} and the second part in Theorem III.1.6.
\end{proof}

For the next lemma, consider the norm on the set of Gibbs measures that generates the {\it strong} topology:
\[
\|\Gamma\| = \sup_{\|f\|_\infty=1} |\Gamma(f)|\ ,
\]
where the supremum is over continuous functions $f$. For a measure $m$ on ${\cal G}$, define $\|m\|$ similarly, but with continuous functions $f$ on ${\cal G}$.
\begin{lem}\label{lem: simon2}
Let $\Gamma,\Gamma' \in {\cal G}$.
\begin{enumerate}
\item If $\Gamma$ is absolutely continuous with respect to $\Gamma'$ then $\frac{d\Gamma}{d\Gamma'}$ is $T_\infty$-measurable, $m_\Gamma$ is absolutely continuous with respect to $m_{\Gamma'}$, and for each $\rho \in ex({\cal G})$,
\[
\frac{dm_\Gamma}{dm_{\Gamma'}}(\rho) = \rho \left( \frac{d\Gamma}{d\Gamma'} \right)\ .
\]
\item The following equation holds:
\[
\|\Gamma-\Gamma'\| = \|m_\Gamma - m_{\Gamma'}\|\ .
\]
\end{enumerate}
\end{lem}

\begin{proof}
The first statement is \cite[Theorem~III.5.1]{Simon} and the second is \cite[Theorem~III.5.2]{Simon}. 
\end{proof}

Next, we notice that for $\rho_1$ and $\rho_2$ in the support of $m_\Gamma$ (for $\Gamma$ in the support of $\kappa_J$),
\begin{eqnarray}\label{eq: overlapsum}
Var_{\rho_1\times\rho_2}\left(\lim_{\Lambda \to \Z^d} \frac{1}{|\Lambda|} \sum_{x \in \Lambda} \sigma_x\sigma'_x\right) &=& \lim_{\Lambda \to \Z^d} \frac{1}{|\Lambda|^2} \sum_{x,y \in \Lambda} \rho_1(\sigma_x\sigma_y)\rho_2(\sigma_x\sigma_y) - \rho_1(\sigma_x)\rho_1(\sigma_y)\rho_2(\sigma_x)\rho_2(\sigma_y) \nonumber \\
&=& \lim_{\Lambda \to \Z^d} \frac{1}{|\Lambda|^2} \sum_{x,y \in \Lambda} \rho_1(\sigma_x\sigma_y)\left[ \rho_2(\sigma_x\sigma_y)-\rho_2(\sigma_x)\rho_2(\sigma_y) \right] \nonumber \\
&+&  \lim_{\Lambda \to \Z^d} \frac{1}{|\Lambda|^2} \sum_{x,y \in \Lambda} \rho_2(\sigma_x)\rho_2(\sigma_y)\left[ \rho_1(\sigma_x\sigma_y)-\rho_1(\sigma_x)\rho_1(\sigma_y) \right]\ .
\end{eqnarray}
We would like to use part 2 of Lemma~\ref{lem: simon} to deduce that this quantity is zero. For this purpose, we introduce the measure $M^*$ on the product space of triples $(J,\{\sigma^i\}_i,(\rho_1,\rho_2))$:
\[
M^*=M\times(m_\Gamma\times m_\Gamma)\ .
\]
Sampling from $M^*$ amounts to sampling a pair $(J,\{\sigma^i\}_i)$ from $M$, noting the value of $\Gamma$, and then sampling two pure states independently from $m_\Gamma$. We must check that this indeed defines a probability measure. For this it suffices to show the following:

\begin{lem}\label{lem: measurable}
For each Borel subset $B$ of ${\cal G}\times {\cal G}$, the map $(J,\{\sigma^i\}_i) \mapsto (m_\Gamma\times m_\Gamma)(B)$ is measurable.
\end{lem}
\begin{proof}
Because the map $(J,\{\sigma^i\}_i) \mapsto \Gamma$ is measurable, it suffices to show that 
\begin{equation}\label{eq: gammamap}
\Gamma \mapsto (m_\Gamma\times m_\Gamma)(B)
\end{equation}
is measurable. We show that the map \eqref{eq: gammamap} is strongly continuous. It is elementary to show that each strongly open set in ${\cal G}$ is a Borel set in the weak topology, so this will prove the lemma. Let $(\Gamma_n)$ be a sequence in ${\cal G}$ that converges strongly to $\Gamma$. Define
\[
\gamma_n = (1/2)(\Gamma_n+\Gamma)\ .
\]
It is not difficult to see that the convergence $\Gamma_n \to \Gamma$ is equivalent to the following:
\begin{equation}\label{eq: equiv}
\int \left| \frac{d\Gamma_n}{d\gamma_n} - \frac{d\Gamma}{d\gamma_n} \right| d\gamma_n \to 0\ .
\end{equation}
Using Lemma~\ref{lem: simon2}, one can show that the quantity $|(m_{\Gamma_n}\times m_{\Gamma_n})(B) - (m_\Gamma\times m_\Gamma)(B)|$ is no bigger than
\[
\int_{{\cal G}} \rho_1 \left(\frac{d\Gamma_n}{d\gamma_n} \right) \int_{{\cal G}} \left| \rho_2 \left(\frac{d\Gamma_n}{d\gamma_n} - \frac{d\Gamma}{d\gamma_n} \right) \right| m_{\gamma_n}(d\rho_2)m_{\gamma_n}(d\rho_1) 
\]
\[
+ \int_{{\cal G}} \rho_2 \left(\frac{d\Gamma}{d\gamma_n} \right) \int_{{\cal G}} \left| \rho_1 \left(\frac{d\Gamma_n}{d\gamma_n} - \frac{d\Gamma}{d\gamma_n} \right) \right| m_{\gamma_n}(d\rho_1)m_{\gamma_n}(d\rho_2)\ .
\]
 
By the definition of the pure state decomposition and the fact that the above derivatives are constant a.s. in each pure state, this equals
\[
2 \int \left| \frac{d\Gamma_n}{d\gamma_n}-\frac{d\Gamma}{d\gamma_n} \right| d\gamma_n\ ,
\]
which converges to 0 by \eqref{eq: equiv}. This proves that \eqref{eq: gammamap} is strongly continuous and completes the proof of Lemma~\ref{lem: measurable}.

\end{proof}

We will now show that for $\rho_1\times\rho_2$-almost all pairs $(\sigma^1,\sigma^2)$ (for $M^*$-almost every $(\rho_1,\rho_2)$), \eqref{eq: overlapsum} equals zero. For a lattice vector $a \in \Z^d$, define the translation $T_a$ on a triple in the support of $M^*$ by
\[
T_a(J,\{\sigma^i\}_i,(\rho_1,\rho_2)) = (T_a(J), \{T_a(\sigma^i)\}_i, (T_a \rho_1, T_a \rho_2))\ ,
\]
where for $i=1,2$, the measure $T_a \rho_i$ is defined as in Section~\ref{subsec: TI}. Since the measure $M$ is translation-invariant and since the pure state decomposition of a state $\Gamma$ is unique, the measure $M^*$ is translation-invariant.

The fact that \eqref{eq: overlapsum} equals zero for $(m_\Gamma \times m_\Gamma)$-almost every pair $(\rho_1,\rho_2)$ (for $\nu(dJ)\times \kappa_J(d\Gamma)$-almost all $\Gamma$) is a straightforward consequence of the following statement. With $M^*$-probability one, if $i=1$ or 2, then
\begin{equation}\label{eq: mixing2}
\lim_{\Lambda \to \Z^d} \frac{1}{|\Lambda|^2} \sum_{x,y \in \Lambda} | \rho_i(\sigma_x\sigma_y) - \rho_i(\sigma_x)\rho_i(\sigma_y)| = 0\ .
\end{equation}

To prove this, let $\varepsilon>0$. For each $(J,\{\sigma^i\}_i,\rho_i)$ and $x \in \Z^d$, we define
\[
N_i(x,\varepsilon) = \min\left\{ N~:~ |\rho_i(\sigma_x\sigma_y)-\rho_i(\sigma_x)\rho_i(\sigma_y)| < \varepsilon \mbox{ for all } y \mbox{ with } |y-x| \geq N \right\}\ .
\]
By Lemma~\ref{lem: simon}, $N_i(x,\varepsilon) < \infty$ for each $i,x$ and $\varepsilon$. For fixed $\varepsilon$, the distribution of the variables $\{N_i(x,\varepsilon)~:~ x \in \Z^d\}$ under $M^*$ is translation-invariant. For a site $x$ define the event $A_{x,\varepsilon}(N)$ that $N_i(x,\varepsilon)\leq N$. Denoting by $I[A_{x,\varepsilon}(N)]$ the indicator of the event $A_{x,\varepsilon}(N)$, we have that by translation-invariance,
\[
F(N,\varepsilon) := \lim_{\Lambda \to \Z^d} \frac{1}{|\Lambda|}\sum_{x \in \Lambda} I[A_{x,\varepsilon}(N)]
\]
exists $M^*$-a.s. Furthermore, for fixed $\varepsilon>0$, $\lim_{N \to \infty} F(N,\varepsilon) = 1$ $M^*$-a.s.

Fix $N$ such that with $M^*$-probability greater than $1-\varepsilon$, there exists a random set $S$ of vertices with (well-defined) density greater than $1-\varepsilon$ such that for each $x\in S$, we have $N_i(x,\varepsilon)\leq N$. Then for $x \in S \cap \Lambda$,
\[
\sum_{y \in \Lambda} |\rho(\sigma_x\sigma_y)-\rho(\sigma_x)\rho(\sigma_y)| \leq 2N^d + \varepsilon |\Lambda|\ ,
\]
so that
\begin{eqnarray*}
\limsup_{\Lambda \to \Z^d} \frac{1}{|\Lambda|^2} \sum_{x,y \in \Lambda} | \rho(\sigma_x\sigma_y) - \rho(\sigma_x)\rho(\sigma_y)| 
&\leq&  \limsup_{\Lambda \to \Z^d} \left[ \frac{1}{|\Lambda|} \sum_{x \in S\cap \Lambda} \left[ \frac{2N^d}{|\Lambda|} + \varepsilon\right] + 2 \frac{|\Lambda \setminus S|}{|\Lambda|} \right] \\
&\leq & 3\varepsilon\ .
\end{eqnarray*}
Therefore with $M^*$-probability at least $1-\varepsilon$, when we replace the limit in \eqref{eq: mixing2} by limsup, it is bounded between 0 and $3\varepsilon$. Letting $\varepsilon \to 0$, we get \eqref{eq: mixing2} with $M^*$-probability 1.


\begin{thebibliography}{1}
\bibitem{AC} Aizenman M., Contucci P., On the Stability of the Quenched state in Mean Field 
Spin Glass Models, {\it J. Stat. Phys.} {\bf 92} (1998) pp. 765-783.

\bibitem{AW} Aizenman M., Wehr J. 
Rounding Effects of Quenched Randomness on First-Order Phase Transitions.
{\it Comm. Math. Phys.} {\bf 130} (1990) pp. 489-528.

\bibitem{aldous} Aldous D., Exchangeability and related topics, in {\it \'Ecole d'\'et\'e de Probabilit\'es de Saint-Flour XIII}, Lecture Notes in Mathematics {\bf 1117}, Springer (1985) pp. 1-198.

\bibitem{lp_sourav} Arguin L.-P., Chatterjee S., Random Overlap Structures: Properties and Applications to Spin Glasses, in preparation.

\bibitem{DS} Dovbysh L., Sudakov V., Gram-de Finettti matrices, {\it J. Soviet. Math.} {\bf 24} (1982) pp. 3047-3054.

\bibitem{Contucci} Contucci P., {Stochastic Stability: A Review and Some Perspectives}, {\it J. Stat. Phys.} {\bf 138} (2010) pp. 543-558.

\bibitem{CG} Contucci P., Giardina C., Spin-Glass Stochastic Stability: a Rigorous Proof, {\it Ann. I. H. Poincar\'e} {\bf 5} (2005) pp. 915-923.

\bibitem{CG_flip} Contucci P., Giardina C., Giberti C., Interaction-Flip Identities in Spin Glasses, {\it J. Stat. Phys.} {\bf 135} (2009) pp. 1181-1203.

\bibitem{kuelske} K\"ulske C., Metastates in Disordered Mean-Field Models: Random Field and Hopfield Models,
{\it J.Stat.Phys.} {\bf 88} (1997) pp. 1257-1293.

\bibitem{Newman} Newman C.M., {\it Topics in Disordered Systems}, Birkh\"auser (1997) 88pp.

\bibitem{CD_NMF} Newman C.M., Stein D.L., Non-Mean-Field Behavior of Realistic Spin Glasses, {\it Phys. Rev. Lett} {\bf 76} (1996) pp. 515-518.

\bibitem{CD_MTC} Newman C.M., Stein D.L., The Metastate Approach to Thermodynamic Chaos, {\it Phys. Rev. E} {\bf 55} (1997) pp. 5194-5211.

\bibitem{TC} Newman C.M., Stein D.L., Thermodynamic Chaos and the Structure of Short-Range Spin Glasses, in {\it Mathematics of Spin Glasses and Neural Networks}, ed. by Anton Bovier and Pierre Picco, Birkha\"user, Boston (1997)  

\bibitem{NS 09} Newman C.M., Stein D.L., Metastates, Translation Ergodicity, and Simplicity of Thermodynamic States in Disordered Systems: An Illustration, in {\it New Trends in Mathematical Physics: Proceedings of the 2006 International Congress of Mathematical Physics}, ed. V. Sidoravicius, Springer (Heidelberg) (2009) pp. 643-652.

\bibitem{panchenko diff} Panchenko D., {On the Differentiability of the Parisi Formula}, {\it Elect. Comm. Prob.} {\bf 13} (2008) pp. 241-247.

\bibitem{panchenko_ds} Panchenko D., On the Dovbysh-Sudakov representation result, preprint (2009), arXiv:0905.1524.

\bibitem{panchenko_GG2} Panchenko D., The Ghirlanda-Guerra identities for mixed p-spin model, {\it C.R.Acad.Sci.Paris, Ser. I} 348 (2010) pp. 189-192.

\bibitem{panchenko_rep} Panchenko D., Spin glass models from the point of view of spin distributions, preprint (2010), arXiv:1005.2720.

\bibitem{Parisi} Parisi G., {Stochastic Stability}, {\it AIP Conf. Proc.} {\bf 553} (2001) pp. 73-79.

\bibitem{Simon} Simon B., {\it The statistical mechanics of lattice gases, vol. 1}, Princeton University Press, Princeton, New Jersey (1993) 520 pp.

\bibitem{Talagrand} Talagrand M., Construction of pure states in mean field models for spin glasses, {\it Prob. Th. Rel. Fields} {\bf 148} (2009) pp. 601-643.

\bibitem{Tal_book} Talagrand M., {\it Spin Glasses: A Challenge for Mathematicians. Cavity and Mean Field Models}, Springer (2003) 586 pp.

\bibitem{limit} van Hemmen J. L.,  Palmer R.G., The thermodynamic limit and the replica method for short-range random systems,  {\it J. Phys. A: Math. Gen.} {\bf 15} (1982) 3881.


\end{thebibliography}
\end{document}